\documentclass[11pt, a4paper]{article}  

\usepackage{hyperref}
\hypersetup{hidelinks}
\usepackage{algpseudocode}
\usepackage{amsthm, bm}

\usepackage{algorithm}
\usepackage{pifont}
\usepackage{amsmath}
\usepackage[nameinlink]{cleveref}
\usepackage{amssymb}

\usepackage{color}
\usepackage{enumerate, enumitem}
\usepackage{mathrsfs} 
\usepackage{titlesec}
\usepackage{mathtools, calc}
\usepackage{verbatim}
\usepackage{listings}
\usepackage{graphicx}
\usepackage{lscape, textcomp}
\usepackage{tikz, subcaption}
\usepackage{xparse}

\usepackage{url}
\usepackage{enumitem}
\usepackage[toc,page]{appendix}
\usetikzlibrary{arrows}
\usepackage{fancyhdr} 
\usepackage{array, stackengine}
\usepackage{pbox}
\usepackage{sectsty}
\usepackage{tablefootnote, longtable}
\usepackage{dsfont}

\sectionfont{\fontsize{12}{15}\selectfont}
\subsectionfont{\fontsize{11}{15}\selectfont}
\subsubsectionfont{\fontsize{11}{15}\selectfont}

\newcolumntype{L}[1]{>{\raggedright\let\newline\\\arraybackslash\hspace{0pt}}m{#1}}
\newcolumntype{C}[1]{>{\centering\let\newline\\\arraybackslash\hspace{0pt}}m{#1}}
\newcolumntype{R}[1]{>{\raggedleft\let\newline\\\arraybackslash\hspace{0pt}}m{#1}}

\definecolor{blue}{rgb}{0.0, 0.0,1.0}
\definecolor{gray}{rgb}{0.7, 0.7, 0.7}
\definecolor{blue-violet}{rgb}{0.54, 0.17, 0.89}

\theoremstyle{plain}
\newtheorem{theorem}{Theorem}[section]
\newtheorem{lemma}[theorem]{Lemma}

\newtheorem{corollary}[theorem]{Corollary}

\theoremstyle{definition}

\newtheorem{definition}[theorem]{Definition}
\newtheorem*{rem*}{Remark}
\newtheorem*{warning*}{Warning}
\newtheorem{rem}[theorem]{Remark}

\DeclareMathOperator{\diag}{diag}

\newsavebox\CBox

\makeatletter
\newcommand{\subalign}[2][c]{%
	\if#1c\vcenter\else\vtop\fi{%
		\Let@ \restore@math@cr \default@tag
		\baselineskip\fontdimen10 \scriptfont\tw@
		\advance\baselineskip\fontdimen12 \scriptfont\tw@
		\lineskip\thr@@\fontdimen8 \scriptfont\thr@@
		\lineskiplimit\lineskip
		\ialign{\hfil$\m@th\scriptstyle##$&$\m@th\scriptstyle{}##$\hfil\crcr
			#2\crcr
		}%
	}%
}
\makeatother

\makeatletter
\renewcommand*{\@fnsymbol}[1]{\ensuremath{\ifcase#1\or *\or \ddagger\or \mathsection\or \vee\or \wedge\or \dagger\or
		\mathsection\or \mathparagraph\or \|\or **\or \dagger\dagger
		\or \ddagger\ddagger \else\@ctrerr\fi}}

\numberwithin{equation}{section}

\begin{document}
\title{Solving a linear program via a single unconstrained minimization}

\author{
        Adilet Otemissov \thanks{Department of Mathematics, School of Sciences and Humanities, Nazarbayev University, Kabanbay Batyr 53, Astana 010000, Kazakhstan; \texttt{aotemissov, alina.abdikarimova@nu.edu.kz}. This research has been funded by Nazarbayev University under the Faculty Development Competitive Research Grant Program for the 2023--2025 Grant \textnumero 20122022FD4138. }
        \and
	Alina Abdikarimova \footnotemark[1]
}

\date{\today}
\maketitle
\footnotesep=0.4cm

{\small
	\begin{abstract}
        This paper proposes a novel approach for solving linear programs. We reformulate a primal-dual linear program as an unconstrained minimization of a convex and twice continuously differentiable merit function. When the optimal set of the primal-dual pair is nonempty, its optimal set is equal to the optimal set of the proposed merit function. Minimizing this merit function poses some challenges due to its Hessian being singular at some points in the domain, including the optimal solutions. We handle singular Hessians using the Newton method with Levenberg-Marquardt regularization. We show that the Newton method with Levenberg-Marquardt regularization yields global convergence to a solution of the primal-dual linear program in at most $O(\epsilon^{-3/2})$ iterations requiring only the assumption that the optimal set of the primal-dual linear program is bounded. Testing on random synthetic problems demonstrates convergence to optimal solutions to very high accuracy significantly faster than the derived worst-case bound. We further introduce a modified merit function that depends on a scalar parameter $\nu > 0$, whose Hessian is nonsingular for all $\nu > 0$ and which reduces exactly to the original merit function when $\nu = 0$. Based on this formulation, we propose a heuristic scheme that performs Newton steps while gradually decreasing $\nu$ toward zero. Numerical experiments indicate that this approach achieves faster convergence, particularly on higher-dimensional problems.

	\end{abstract}
	
	\bigskip
	
	\begin{center}
	\end{center}
}

\maketitle

\section{Introduction} \label{sec:intro}
In this paper, we intend to solve primal and dual linear programs in standard form:
\vspace{\baselineskip}

	\noindent\begin{minipage}{.5\linewidth}
		\begin{equation}\label{eq:primal}
			\tag{P}
			\begin{alignedat}{4}
				\text{min} \;\; & & c^Tx & \\
				\text{subject to} \;\;&     &   Ax & = b \\
				& & x& \geq 0
			\end{alignedat}
		\end{equation}
	\end{minipage}%
	\begin{minipage}{.5\linewidth}
		\begin{equation}\label{eq:dual}
			\tag{D}
			\begin{alignedat}{4}
				\text{max} \;\; & & b^T\lambda \;\;\;\;\;\;& \\
				\text{subject to} \;\;&     &   A^T\lambda + s & = c \\
				& & s& \geq 0,
			\end{alignedat}
		\end{equation}
	\end{minipage}
\vspace{\baselineskip}

\noindent where $x, s \in \mathbb{R}^n$ and $\lambda \in \mathbb{R}^m$ are the variables and where $c \in \mathbb{R}^n$, $b \in \mathbb{R}^m$ and $A \in \mathbb{R}^{m \times n}$ ($m<n$) are given. 

Linear programming problems are widely encountered in applications such as engineering, finance, transport and many more \cite{Hillier2014}. Linear programs are also commonly used as subproblems to obtain solutions to more complex problems \cite{Wolsey2020}. In many of these cases, there is a demand to obtain accurate solutions to linear programs in a fast and reliable way. Meeting this demand becomes increasingly challenging as problem sizes in practical applications grow bigger. These challenges are motivating researchers to modify the current methods or seek alternative ways to solve linear programs (see, e.g., \cite{Chowdhury2022, Vu2018, Wu2023}).

\paragraph{Historical overview.} Linear Programming is one of the success stories in optimization. Its story began in 1950's with the advent of simplex method discovered by  Dantzig~\cite{Dantzig1951_1,Dantzig1951_2}. It has been developed and modified since and, nowadays, it is one of the state-of-the-art methods for solving linear programs \cite{Nesterov2006,Vanderbei2020}. Simplex method performs very well in practice, however, it has not been proven yet that simplex method can solve every linear program in polynomial time. For each version of simplex, one can find instances where simplex method will spend exponential number of iterations before reaching an optimal solution \cite{Klee1972, Vanderbei2020}. Later in 1979, Khachiyan~\cite{Khachiyan1979} proposed a new method for solving linear programs called ellipsoid method. It was proven that ellipsoid method can solve linear programs in polynomial time, but it turned out to be slow in practice. Polynomial time convergence of ellipsoid method was a significant theoretical breakthrough as it indicated that there could be an algorithm that can provably solve linear programs in polynomial time and perform well in practice. Such algorithm came to light in 1984 in the famous paper by Karmarkar~\cite{Karmarkar1984} who proposed a new polynomial time algorithm that is now classified as an Interior Point Method (IPM) (see also \cite{Nesterov1994}). The discovery of this method revitalized research in linear programming resulting in fast development of IPMs. Nowadays, IPMs are one of the state-of-the-art methods for solving linear programs that perform very well in practice matching performance of simplex method.

\subsection{Reformulations of linear programs}
\paragraph{Barrier reformulation.} A starting point in the derivation of an IPM is the barrier formulation of \eqref{eq:primal}:
\begin{equation}\label{eq:barrier_problem}
			\begin{alignedat}{4}
				\text{min} \;\; & c^Tx - \nu \sum_{j = 1}^n \log(x_j)   \\
				\text{subject to} \;\;&  Ax = b, 
			\end{alignedat}
		\end{equation}
whose Lagrangian function is 
$$L(x,\lambda) = c^Tx - \nu \sum_{j = 1}^n \log(x_j) + \lambda^T (b-Ax). $$
Taking the derivatives with respect to $x_j$'s and $\lambda_i$'s yield the first-order optimality conditions for $L(x,\lambda)$:
\begin{equation} \label{eq:comp_slackness}
        \begin{aligned}
            x_js_j & = \nu \;\; \text{for $j= 1, \dots, n$,}  \\
            Ax - b & = 0, \\
            A^T\lambda + s - c& = 0.\\
            \end{aligned}
\end{equation} 
If we add nonnegativity constraints to \eqref{eq:comp_slackness} and set $\nu = 0$ we will end up with optimality conditions for \eqref{eq:primal} and \eqref{eq:dual}. However, adding nonnegativity constraints to \eqref{eq:comp_slackness} will complicate the nonlinear system. Instead, one can drive the iterates towards satisfaction of the optimality conditions by solving the nonlinear system \eqref{eq:comp_slackness} using Newton method with $\nu$ approaching zero, while keeping $x$ and $s$ positive. The solution to \eqref{eq:comp_slackness} for different values of $\nu$ creates a central path $\{ (x_{\nu}, \lambda_{\nu}, s_{\nu}): \nu>0\}$ which leads us to an optimal solution.  Solution of \eqref{eq:comp_slackness} with Newton method for different values of $\nu$ generates a sequence of iterates that are not exactly on the path but are close to it. As $\nu $ approaches zero, barrier problem \eqref{eq:barrier_problem} better approximates original problem \eqref{eq:primal} and the iterates approach the boundary of the feasible region but never cross it because of the logarithmic term. 

It should be noted that popular versions of IPMs allow iterates to be infeasible with respect to primal and dual equality constraints; they require only that the iterates $x^k$ and $s^k$ remain strictly positive. 

\paragraph{Penalty reformulation.} In the barrier formulation, the nonnegativity constraints are handled through logarithmic terms, that create a `barrier' that the iterates are never allowed to cross. In a penalty formulation the iterates are allowed to cross into the infeasible region, but are incurred a large cost if they do so. The simplest and common penalty formulation for
\begin{equation}\label{eq:equal_const_problem}
    \min \, f(x) \text{ subject to } g_i(x) = 0 \text{ for $i = 1, \dots, m$}
\end{equation}
is a quadratic penalty formulation
\begin{equation}\label{eq:penalty_formulation}
    \min \; h_{\nu}(x) := f(x) + \frac{1}{2\nu} \sum_{i=1}^m g_i^2(x),
\end{equation}
where $\nu > 0$ is the penalty parameter. Driving $\nu$ towards zero, we increase the penalty for violating the constraints. Given a sequence of penalty parameters $\nu_k$ that gradually approach zero, we can solve \eqref{eq:penalty_formulation} for each value of $\nu^k$ and feed the computed solution as a starting point to minimize $h_{\nu^{k+1}}(x)$. Nocedal and Wright showed in  
\cite[Theorem 17.1]{Nocedal2006} that by doing so we generate a sequence of solutions $x^k$ that approach an optimal solution $x^*$ of the original problem as $\nu^k \rightarrow 0$. However, as $\nu^k$ becomes smaller, problem \eqref{eq:penalty_formulation} becomes increasingly ill-conditioned \cite{BenTal2023, Nocedal2006}. This issue can be alleviated with the use of an augmented Lagrangian formulation instead of \eqref{eq:penalty_formulation}, which can force satisfaction of the constraints at a non-zero value of the penalty parameter. Augmented Lagrangian formulation for linear programs was used, for example, in `Idiot' crash \cite{Galabova2020, Forrest2024} to obtain approximate solutions to linear programs (see also \cite{Evtushenko2005, Guler1992, Mangasarian2004}). This solution can then be used as a good starting point for simplex method to speed up convergence. One of the main drawbacks of a typical quadratic penalty method is its inability to produce an optimal solution to the original problem in a single minimization; it usually requires solving several optimization problems. 

If, in addition to equality constraints, we have nonnegativity constraints $x \geq 0$, we can use the following formulation 
\begin{equation}\label{eq:ineq_constr_problem}
    \min \; h'_{\nu}(x) := f(x) + \frac{1}{2\nu} \sum_{i=1}^m g_i^2(x) + \frac{1}{2\nu} \sum_{j=1}^n \max\{ -x_j,0 \}^2.
\end{equation}
When $x_j$ is negative, then the penalty term $\max\{ -x_j,0 \}$ becomes positive, discouraging $x$ from being optimal at a solution with negative entries. The penalty term $\max\{ -x_j,0 \}^2$ is continuously differentiable, but is not twice continuously differentiable prohibiting the use of classical second-order methods. 

Recall that a penalty function is called exact if there exists a nonzero value of $\nu$ such that a single minimization of the penalty function recovers an exact solution of the original problem. It is well known that, in general, if penalty terms in the penalty function are continuously differentiable as in \eqref{eq:ineq_constr_problem}, then the penalty function is not exact (see \cite{Dolgopolik2016, Nocedal2006}). In general, a nondifferentiability is a necessary requirement for the penalty function to be exact. 
\subsection{Our contributions}

Are nondifferentiabilities unavoidable if we want to solve a problem with inequality constraints in a single minimization of a penalty function? In this regard,  Bertsekas~\cite[p.~1]{Bertsekas1975} writes ``Except for trivial cases, nondifferentiabilities are a necessary evil if the penalty method is to yield an optimal solution in a single minimization.'' However, this holds true if the penalty function has the commonly used form: `objective function $+$ penalty parameter $\times$ constraint violation' (see \cite[Eq.~(5)]{Bertsekas1975}). Our paper demonstrates that we can avoid first- and second-order nondifferentiabilities in an exact reformulation of a linear program if we step away from the convention of including the objective function as a separate term in the reformulation. In \Cref{sec:future_work}, we show that this holds true for a more complex constrained optimization problem.

We propose a convex and twice continuously differentiable merit function $f_q(x,\lambda,s)$ defined in \eqref{eq:merit_function} in \Cref{sec:penalty_formulation} that combines the ideas of barrier and penalty methods. Our formulation relies on optimality conditions \eqref{eq:comp_slackness}, but with complementary slackness $x_js_j$ ($1 \leq j \leq n$) replaced by the duality gap $c^Tx- b^T\lambda$. Nonnegativity constraints are imposed through an addition of max functions, just like in \eqref{eq:ineq_constr_problem}, but raised to a power greater than 2 to make the max penalty terms twice continuously differentiable. For example, for $x_j$, we have $\max\{-x_j,0\}^q$ in the merit function, where $q>2$.

\paragraph{Strengths.} Our formulation \eqref{eq:merit_function} has the following strengths:
\begin{enumerate}
    \item \eqref{eq:merit_function} is an unconstrained optimization problem.
    \item Merit function $f_q$ in \eqref{eq:merit_function} is convex and twice continuously differentiable. 
    \item When primal-dual pair \eqref{eq:primal} and \eqref{eq:dual} have an optimal solution, the optimal set of $f_q$ is equal to the optimal set of \eqref{eq:primal} and \eqref{eq:dual}. 
    \item When primal-dual pair \eqref{eq:primal} and \eqref{eq:dual} have an optimal solution, the minimum of $f_q$ is equal to zero, in other words, \eqref{eq:merit_function} is a zero residual problem. 
    \item Based on the computed value of $\min f_q$ we can conclude whether \eqref{eq:primal} and \eqref{eq:dual} have an optimal optimal or unbounded/infeasible. 
    \item \eqref{eq:merit_function} is free of a penalty parameter that needs to be increased or decreased, as usually done in a typical penalty method.
    \item Merit function $f_q$ is defined over the whole Euclidean space and so it allows initialization of a starting point at any point in $\mathbb{R}^{2n+m}$ including those that are infeasible for \eqref{eq:primal} and \eqref{eq:dual}.
    \item For $q = 3$, the Hessian of $f_q$ is Lipschitz continuous with Lipschitz constant~1. 
\end{enumerate}

\paragraph{Weaknesses.} Our formulation \eqref{eq:merit_function} has the following weaknesses:
\begin{enumerate}
    \item The Hessian of $f_q$ is singular at some points in the domain and when \eqref{eq:primal} and \eqref{eq:dual} have an optimal solution, the Hessian is singular at any optimal solution of the merit function.
    \item The Hessian has size $(2n+m) \times (2n+m)$, in contrast to the smaller ($m \times m$) systems typically solved in IPMs. 
    \item Although the computed value of $\min f_q$ can tell us whether \eqref{eq:primal} and \eqref{eq:dual} do not have an optimal solution, i.e., whether they are unbounded/infeasible, it cannot distinguish between unbounded and infeasible cases. 
    
\end{enumerate}
\paragraph{Addressing the first weakness.} Since $f_q$ is convex and twice continuously differentiable, we can use second-order methods to minimize $f_q$. However, these methods need to be applied with special care since the Hessian of $f_q$ is singular at the optimal solution and at some other points in the domain. In particular, we use Newton method with Levenberg-Marquardt regularization \cite{Levenberg1944, Marquardt1963}. For any optimal solution $(x^*,\lambda^*,s^*)$ of primal-dual pair \eqref{eq:primal} and \eqref{eq:dual}, we show that Newton method with Levenberg-Marquardt regularization achieves $|f_q(x^k, \lambda^k, s^k) - f_q(x^*,\lambda^*,s^*)| < \epsilon$ globally it at most $O(\epsilon^{-3/2})$ iterations\footnote{In one iteration we solve one linear system to find a search direction.} provided that the optimal set of \eqref{eq:primal} and \eqref{eq:dual} is bounded. 

In practice, Newton method with Levenberg-Marquardt regularization may converge slowly. To improve convergence, we propose a heuristic approach (see \Cref{alg:Newton_for_HMF}) that minimizes a modified merit function $h_{q, \nu}$ defined in \eqref{eq:homotopy_merit_function}, which involves an additional parameter $\nu \geq 0$. This new merit function can be thought of as an approximation to $f_q$. The approximation is better for smaller $\nu$ and, at $\nu = 0$, $h_{q, \nu}$ is identically equal to $f_q$. The modified merit function has an advantage of having a nonsingular Hessian for $\nu > 0$. At each iteration we solve a nonsingular Newton system to find a suitable step that minimizes $h_{q,\nu^{k}}$ and then we set $\nu^{k+1} = \theta \nu^k$ where $\theta$ is strictly less than 1 to drive $\nu^k$'s to zero.  

We test the two approaches (see \Cref{alg:Newton_with_LM} and \Cref{alg:Newton_for_HMF}) on synthetically generated random problems of varying sizes. Both approaches converge to optimal solutions of all three test problems with the heuristic approach (\Cref{alg:Newton_for_HMF}) performing better than Newton method with Levenberg-Marquardt regularization (\Cref{alg:Newton_with_LM}) for a larger dimensional problem.

\paragraph{Addressing the second weakness.} Much of computational effort at each iteration of \Cref{alg:Newton_with_LM} and \ref{alg:Newton_for_HMF} is taken up in solving a linear system to find a Newton step. Typically, solution of a linear system involves Cholesky factorization of the Hessian. Since the Hessian has size $(2n+m) \times (2n+m)$, a naive application of Cholesky factorization requires $O((2n+m)^3)$ floating operations. In \Cref{sec:linear_system}, we show that we can solve the linear system by performing two separate Cholesky factorizations of smaller $m\times m$ and $n \times n$ systems, thus reducing computational burden from $O((2n+m)^3)$ to $O(m^3+n^3)$.

\paragraph{Addressing the third weakness.} If after computing the minimum of the merit function $f_q(x,\lambda, s)$ we have $\min f_q(x,\lambda, s) > 0$ (see \Cref{cor:no_opt_sol}) we conclude that \eqref{eq:primal} and \eqref{eq:dual} have no optimal solution. However, we cannot, with certainty, say which of the problems \eqref{eq:primal} and \eqref{eq:dual} is unbounded and which is infeasible, or whether they are both infeasible. To reliably distinguish between the two cases, we suggest using the simplified homogeneous formulation \eqref{eq:hlf} presented in \cite{Xu1996} (see also \cite{Ye1994}). Based on the solution of the homogeneous formulation, we can reliably conclude whether the primal \eqref{eq:primal}, the dual \eqref{eq:dual}, or both are infeasible.  

\subsection{Outline} In \Cref{sec:penalty_formulation}, we define the merit function and describe its properties in relation to primal-dual pair \eqref{eq:primal} and \eqref{eq:dual}. We derive the Hessian of the merit function and show that it has two notable properties: i) it is Lipschitz continuous, but ii) singular at some points in the domain of the merit function.
\Cref{sec:reg_Newton} discusses regularized Newton methods focusing on Levenberg-Marquardt regularization. Using existing global convergence results (see \cite{Mishchenko2023}), we prove that Newton method with Levenberg-Marquardt regularization provides a globally convergent scheme for finding an optimal solution of \eqref{eq:primal} and \eqref{eq:dual} with $O(\epsilon^{-3/2})$ global convergence rate. At the end of \Cref{sec:reg_Newton}, we show that we can solve a linear system, needed to determine a Newton search direction, more efficiently avoiding a full Cholesky factorization of the Hessian. \Cref{sec:numerics} begins with the introduction of a new merit function --- a modification of the old one --- but that has better properties, for example, its Hessians are nonsingular. 
With this new merit function, we propose a heuristic approach that numerically performs better. In \Cref{sec:numerics}, we test both merit functions on synthetically generated random linear programs that have optimal solutions. We also test both formulations on a random unbounded linear program. \Cref{sec:conclusion} concludes our work and discusses future work.

\subsection{Notation and definitions} We use 
\begin{equation}\label{eq:relu_def}
    (-x)_+^q := \begin{pmatrix}
            \max\{-x_1,0\}^q & \dots & \max\{-x_n,0\}^q
        \end{pmatrix}^T
\end{equation}
where 
$$ \max\{-x_j,0\}^q = \begin{cases}
    (-x_j)^q & \text{if $x_j<0$,} \\
    0 & \text{if $x_j\geq 0$.}
\end{cases}$$
We write $\diag(x)$ to denote a diagonal matrix with entries of $x$ on the diagonal. We use $0_n$ to denote a zero vector of size $n$ where the specification is needed. We use $\|\cdot\|_2$ and $\|\cdot\|_F$ for the Euclidean norm and the Frobenius norm, respectively.

\begin{definition} \label{def:Lipschitz}
    For a twice continuously differentiable function $f: \mathbb{R}^n \rightarrow \mathbb{R}$, we say that the Hessian $\nabla^2 f(x)$ is $L$-Lipschitz if and only if 
    \begin{equation}\label{eq:Lipschitz}
    \text{$\|\nabla^2 f(x) - \nabla^2 f(y)\|_2 \leq L \| x-y \|_2$ for all $x, y \in \mathbb{R}^n$.}
    \end{equation}
\end{definition}
\begin{definition}\label{def:bounded_set}
    Let $X$ be a set of points in $\mathbb{R}^n$. We say that $X$ is bounded if and only if there exists a $D>0$ such that, for every point $x$ in $X$, $\|x\|_2 \leq D$. 
\end{definition}


\section{The merit function}\label{sec:penalty_formulation}
Recall that $(x,\lambda,s)$ is an optimal solution to the primal-dual pair \eqref{eq:primal} and \eqref{eq:dual} if and only if $(x,\lambda,s)$ satisfy
    \begin{equation} \label{eq:optimality_conditions}\tag{OC}
        \begin{aligned}
            c^Tx - b^T\lambda & = 0  \\
            Ax - b & = 0 \\
            A^T\lambda + s - c& = 0\\
            x, s & \geq 0.
        \end{aligned}
    \end{equation} 
We propose to find a solution to \eqref{eq:optimality_conditions} by finding a minimizer to the following merit function
\begin{equation} \label{eq:merit_function}\tag{MF}
\begin{aligned}
    \mathop{\arg\min}_{(x,\lambda,s) \in \mathbb{R}^{2n+m}}\, f_q(x,\lambda,s) := 
    & \; \frac{1}{2}(c^Tx - b^T\lambda)^2 + \frac{1}{2}\|Ax - b\|_2^2 + \frac{1}{2}\|A^T\lambda + s - c\|_2^2 \\
    & + \frac{1}{q(q-1)}\sum_{j=1}^n \left( \max\{-x_j,0\}^q + \max\{-s_j,0\}^q \right),
\end{aligned}
\end{equation}
where $q > 2$ is fixed. Note that the merit function $f$ is convex and is twice continuously differentiable for $q>2$ (\cite[p.~59]{Chong2001})\footnote{Note that the requirement $q>2$ is needed for $\max\{-x_j,0\}^q + \max\{-s_j,0\}^q$ to be twice continuously differentiable.}. 

Suppose that an optimal solution to \eqref{eq:primal} and \eqref{eq:dual} exists and let $(x^*,\lambda^*,s^*)$ be any minimizer of \eqref{eq:primal} and \eqref{eq:dual}. It is clear that optimality conditions \eqref{eq:optimality_conditions} are satisfied for $(x^*,\lambda^*,s^*)$ if and only if $f_q(x^*,\lambda^*,s^*) = 0$. Given that $f_q(x,\lambda,s)$ is nonnegative for all $(x,\lambda,s) \in \mathbb{R}^{2n+m}$, this also implies that $(x^*,\lambda^*,s^*)$ is an optimal solution to $\eqref{eq:merit_function}$. In other words, the optimal set of primal-dual pair \eqref{eq:primal} and \eqref{eq:dual} is equal to the optimal set of \eqref{eq:merit_function} when \eqref{eq:primal} and \eqref{eq:dual} have an optimal solution. We formally state these simple observations in the following theorem. 
\begin{theorem}\label{thm:equivalence}
    A solution $(x^*,\lambda^*,s^*)$ is optimal for primal-dual pair \eqref{eq:primal} and \eqref{eq:dual} if and only if $(x^*,\lambda^*,s^*)$ is an optimal solution to \eqref{eq:merit_function} and $f_q(x^*,\lambda^*,s^*) = 0$.
\end{theorem}
\begin{rem}
    When \eqref{eq:primal} and \eqref{eq:dual} have an optimal solution, we can view \eqref{eq:merit_function} as a zero residual problem. 
\end{rem}

When \eqref{eq:primal} and \eqref{eq:dual} are unbounded and/or infeasible, there is no solution that satisfies \eqref{eq:optimality_conditions}. In this case, the merit function $f_q$ cannot have a finite optimum $(x^*,\lambda^*,s^*)$ at which it is equal to zero, since if there were, then this solution must satisfy \eqref{eq:optimality_conditions}. We state this result as a corollary of \Cref{thm:equivalence}. 
\begin{corollary}\label{cor:no_opt_sol}
    Primal-dual pair \eqref{eq:primal} and \eqref{eq:dual} are each either infeasible or unbounded if and only if $\min_{(x,\lambda,s) \in \mathbb{R}^{2n+m}} f_q(x,\lambda,s) > 0$, where $f_q(x,\lambda,s)$ is defined in \eqref{eq:merit_function}.
\end{corollary}
\begin{proof} The result follows from the fact that $f_q(x,\lambda,s) \geq 0$ for all $(x,\lambda,s) \in \mathbb{R}^{m+2n}$ and \Cref{thm:equivalence}.
\end{proof}

\begin{rem}
    The merit function is defined over the entire Euclidean space allowing iterates to be infeasible with respect to primal and dual constraints.
\end{rem}

\subsection{Detecting infeasibility}
\Cref{cor:no_opt_sol} can help us identify the cases when the primal and dual problems are infeasible and/or unbounded. If, during solution of \eqref{eq:merit_function}, the gradient of $f_q$ approaches zero, but the merit function itself fails to do so, this might be an indication that \eqref{eq:primal} and \eqref{eq:dual} do not have an optimal solution. One of the pitfalls of using this criterion, however, is its inability to distinguish between unbounded and infeasible problems. To distinguish between the two cases, we suggest using the simplified homogeneous formulation of a linear program presented in \cite{Xu1996} (see also \cite{Ye1994}):
\begin{equation}\label{eq:hlf}
    \begin{aligned}
        c^Tx-b^T\lambda + \kappa & = 0, \\
        Ax - b\tau & = 0, \\
        A^T\lambda+s - c\tau & = 0, \\
        x, s, \tau, \kappa & \geq 0.
    \end{aligned}
\end{equation}
The simplified homogeneous formulation has been one of the reliable ways to detect infeasibility in IPMs \cite{Andersen2000, Wright1997}. Its solution can reveal whether the problem is primal infeasible, dual infeasible, or both. Simplified homogeneous formulation \eqref{eq:hlf} is satisfied for a trivial solution $(x, \lambda, s, \tau,  \kappa) = (0,0,0,0,0)$, but it can be shown that \eqref{eq:hlf} also has a strictly complementary solution $(\hat{x},  \hat{\lambda}, \hat{s}, \hat{\tau}, \hat{\kappa})$, which satisfies $\hat{\tau} \hat{\kappa} = 0$ with $\hat{\tau} + \hat{\kappa} > 0$ and
\begin{equation}
    \text{$\hat{x}_j \hat{s}_j = 0$ and $\hat{x}_j + \hat{s}_j > 0$ for $j = 1, 2, \dots, n$.}
\end{equation}
Having computed a strictly complementary solution $(\hat{x},  \hat{\lambda}, \hat{s}, \hat{\tau}, \hat{\kappa})$, we can compute the optimal solution for \eqref{eq:primal} and \eqref{eq:dual} or determine whether they are infeasible based on the following criteria.
\begin{enumerate}
    \item $\hat{\tau} > 0$ if and only if \eqref{eq:primal} and \eqref{eq:dual} have an optimal solution. The optimal solution is given by $(x^*,\lambda^*,s^*) = (\hat{x}/\hat{\tau}, \hat{\lambda}/\hat{\tau}, \hat{s}/\hat{\tau})$. Note that, in this case, $\hat{\kappa} = 0$ due to complementarity.
\end{enumerate}
If $\hat{\kappa} > 0$ then at least one of $-c^T\hat{x}$ and $b^T\hat{\lambda}$ must be positive. 
\begin{enumerate}
  \setcounter{enumi}{1}
    \item If $\hat{\kappa} > 0$ and if $-c^T \hat{x} > 0$ then \eqref{eq:dual} is infeasible. 

    \item If $\hat{\kappa} > 0$ and if $b^T \hat{\lambda} > 0$ then \eqref{eq:primal} is infeasible. 
\end{enumerate}
For the proofs of the above results and for more details, refer to \cite{Andersen2000,Wright1997,Xu1996}. 

We formulate a merit function for the simplified homogeneous formulation in \Cref{sec:mf_for_hlf}. The new merit function is not much more complicated than \eqref{eq:merit_function}, as it includes only two additional variables, and can be solved using the same approach (see \Cref{alg:Newton_with_LM} and \ref{alg:Newton_for_HMF}) as the one we use to solve \eqref{eq:merit_function}.

Having said the above, when testing \Cref{alg:Newton_with_LM} and \ref{alg:Newton_for_HMF} on an unbounded problem, we could correctly identify the problem to be unbounded without relying on the simplified homogeneous formulation. See \Cref{sec:unb_lin_prog} for more details.



\subsection{The Hessian of the merit function}

Second-order methods are a standard approach for solving \eqref{eq:merit_function} and, therefore, it is imperative to understand the properties of the Hessian of the merit function. In this section, we show that the Hessian of $f_q(x,\lambda, s)$ has two notable properties: i) it is Lipschitz continuous, but ii) singular at some points in the domain of $f_q$.

Let us first define 
\begin{equation}
    \begin{aligned}
        \gamma & :=  c^T x - b^T\lambda \\
        \rho & := b - Ax \\
        \sigma & := c - A^T\lambda - s 
    \end{aligned}
\end{equation}
The gradient and Hessian of $f_q(x,\lambda,s)$ are given by 
$$\nabla f_q(x, \lambda, s) = \gamma \begin{pmatrix}
    c \\ -b \\ 0
\end{pmatrix} - \begin{pmatrix}
    A^T\rho \\ A\sigma \\ \sigma 
\end{pmatrix} - \frac{1}{q-1} \begin{pmatrix}
    (-x)_+^{q-1} \\ 0_m \\ (-s)_+^{q-1}
\end{pmatrix} $$
and
$$ \nabla^2 f_q(x, \lambda, s) = H + \diag((-x)_+^{q-2}, 0_m , (-s)_+^{q-2}), $$
where  
\begin{equation}\label{eq:H}
    H = \begin{pmatrix}
    cc^T+A^TA & -cb^T & 0 \\
    -bc^T & bb^T+AA^T & A \\
    0 & A^T & I 
\end{pmatrix},
\end{equation}
and where $(-x)_+^{q}$ is defined in \eqref{eq:relu_def}. Note that if $(x,\lambda,s)$ satisfies optimality conditions \eqref{eq:optimality_conditions} then the gradient at $(x,\lambda,s)$ is identically zero. However, the converse is not always true, for example, when \eqref{eq:primal} is unbounded. 

For $(x, s) \geq 0$, the Hessian is singular as it is equal to $H$, whose rank is at most $m+n+1$.
\begin{lemma}
    The rank of matrix $H$ defined in \eqref{eq:H} is at most $m+n+1$.
\end{lemma}
\begin{proof}
    We can express $H$ as a sum of two matrices, whose ranks are at most $1$ and $m+n$:
    $$H = 
    R^TR + \begin{pmatrix}
    A^TA & 0  \\
    0 & B^TB
\end{pmatrix},$$
where $R = \begin{pmatrix}
    c^T & -b^T & 0^T_n
\end{pmatrix}$ and $B = \begin{pmatrix}
    A^T & I
\end{pmatrix}$. The result follows from the subadditivity of rank.
\end{proof}

\begin{rem}
    Suppose that \eqref{eq:primal} and \eqref{eq:dual} have an optimal solution $(x^*, \lambda^*, s^*)$. By \Cref{thm:equivalence}, $(x^*, \lambda^*, s^*)$ must be one of the minimizers of $f_q(x, \lambda, s)$. Since $(x^*, \lambda^*, s^*) \geq 0$, $\nabla^2 f_q(x^*, \lambda^*, s^*) = H$. In other words, the Hessian $\nabla^2 f_q(x^*, \lambda^*, s^*)$ is singular at any optimal solution $(x^*, \lambda^*, s^*)$ of \eqref{eq:primal} and \eqref{eq:dual}.
\end{rem}

The good news is that the Hessian is Lipschitz continuous for $q = 3$ (see \Cref{def:Lipschitz}). 
\begin{lemma}\label{lemma:Hessian_is_Lipschitz}
    For $q = 3$, the Hessian $\nabla^2 f_q(x, \lambda, s)$ is $1$-Lipschitz. 
\end{lemma}
\begin{proof}
    For any $(x,\lambda,s)$ and $(x',\lambda',s')$ we have
\begin{equation}
    \begin{aligned}
        \| \nabla^2f_3(x,\lambda,s) - \nabla^2f_3(x',\lambda',s') \|_2^2 & = \| \diag((-x)_+-(-x')_+, 0_m , (-s)_+-(-s')_+)) \|_2^2 \\
        & \leq \| \diag((-x)_+-(-x')_+, 0_m , (-s)_+-(-s')_+)) \|_F^2 \\
        & \leq \sum_{j=1}^n (x_j - x'_j)^2+ (s_j - s'_j)^2 \\
        & \leq \sum_{j=1}^n (x_j - x'_j)^2+ (s_j - s'_j)^2 + \sum_{i=1}^m (\lambda_j - \lambda'_j)^2,
    \end{aligned}
\end{equation}
where the penultimate inequality follows from $|\max\{-y,0\} - \max\{-y',0\}| \leq |y - y'|$ for any scalars $y$ and $y'$.
\end{proof}




\section{Regularized Newton method to minimize the merit function}\label{sec:reg_Newton}

Classical methods to solve \eqref{eq:merit_function} include gradient descent, Newton, Gauss-Newton, quasi-Newton and trust-region methods \cite{Boyd2004, Conn2000, Nocedal2006}. Gradient descent method is known to perform poorly for convex functions with even mildly ill-conditioned Hessians \cite{Boyd2004}, let alone singular Hessians. 
Functions with singular Hessians pose challenges for second-order methods, too. For example, defining Newton search direction, which is determined by solving a linear system that involves the Hessian, becomes impossible. Therefore, second-order methods must be applied with special care when minimizing functions with singular Hessians.

Numerous methods have been proposed for minimizing convex functions with singular Hessians such as cubic regularization \cite{Cartis2010, Cartis2011_1, Cartis2011_2, Hanzely2020, Nesterov2006, Yue2019}, trust-region-based \cite{Conn2000, Fan2014, Ueda2014, Wang2023} and Levenberg-Marquardt \cite{Levenberg1944, Marquardt1963, Polyak2009, Ueda2010, Yamashita2001} methods. Below, we briefly describe the Levenberg-Marquardt approach and, using existing results, in 
\Cref{thm:MF_global_convergence} we prove that Newton method with Levenberg-Marquardt regularization (\Cref{alg:Newton_with_LM}) converges globally to an optimal solution of the primal-dual pair \eqref{eq:primal} and \eqref{eq:dual} in at most $O(\epsilon^{-3/2})$ iterations.

For details on cubic regularization and trust-region-based approaches, we refer the reader to the respective literature. 

%

\paragraph{Classical Newton method.} Recall that, for a twice continuously differentiable function $f$ at iterate $x^k$, a standard Newton step is defined as
$$x^{k+1} = x^k - \alpha^k \nabla^2 f(x^k)^{-1} \nabla f(x^k),$$
where $\alpha^k$ is the step length that is commonly defined by Armijo backtracking, Wolfe, Goldstein conditions or set to 1 for all $k \geq 1$, in which case it is pure Newton \cite{Nocedal2006}. For strongly convex $f$, $\nabla^2 f(x^k)$ is positive definite so that $-\nabla^2 f(x^k)^{-1} \nabla f(x^k)$ is a descent direction at all iterations. For strongly convex functions, Newton method with backtracking line search is globally convergent and, with the additional assumption of a Lipschitz continuous Hessian, one can show that it enjoys local quadratic convergence to the global minimizer \cite{Boyd2004}. For general convex functions, however, such properties are not guaranteed since Newton step may not even be well defined as the inverse of $\nabla^2 f(x^k)$ may not exist. When this happens for a convex $f$, it implies that some eigenvalues of its Hessian $\nabla^2 f(x^k)$ are equal to zero, which means that $f$ is locally flat in certain directions. This is the case for our merit function $f$ in \eqref{eq:merit_function}. For $(x, s) \geq 0$, the Hessian $\nabla^2 f_q(x, \lambda, s)$ has rank at most $m+n+1$ and therefore $f_q(x, \lambda, s)$ is locally constant along a subspace of dimension at least $n-1$.

\subsection{Levenberg-Marquardt regularization}

\begin{algorithm}[t]
\caption{Newton method with Levenberg–Marquardt regularization}
\label{alg:Newton_with_LM}
\begin{algorithmic}[1]
\State Initialize $x^0, s^0 \in \mathbb{R}^n$, $\lambda^0 \in \mathbb{R}^m$ and $\mu^0 > 0$.
\For{$k = 0, 1, \dots$}
    \State Let $\nabla^2 f^k := \nabla^2 f_q(x^k, \lambda^k, s^k)$ and $\nabla f^k := \nabla f_q(x^k, \lambda^k, s^k)$, where $f_q$ is defined in \eqref{eq:merit_function}.
    \State Update
    \begin{equation}\label{eq:mf_update}
        \begin{pmatrix}
        x^{k+1} \\
        \lambda^{k+1} \\
        s^{k+1}
    \end{pmatrix}
    =
    \begin{pmatrix}
        x^k \\
        \lambda^k \\
        s^k
    \end{pmatrix}
    - \alpha^k (\nabla^2 f^k + \mu^k I)^{-1} \nabla f^k,
    \end{equation}
    \State where $\alpha^k$ is equal to $1$ or chosen using a line search method such as Armijo backtracking.
    \State Update $\mu^{k+1}$.
\EndFor
\end{algorithmic}
\end{algorithm}

In the Levenberg-Marquardt approach (see \Cref{alg:Newton_with_LM}), the Hessian is modified by adding $\mu^k I $ for some positive parameter $\mu^k$, rendering the Hessian positive definite. When $\mu^k$ is small but positive, the Newton step is well defined and the modified Hessian does not deviate much from the original Hessian. The sequence of parameters $\mu^k$ can be defined in various ways. Perhaps the simplest idea is to set $\mu^k$ to some small scalar $\mu > 0$ so that the modified Hessian is sufficiently positive definite \cite[Section 3.4]{Nesterov2006}. One can prove that Levenberg-Marquardt with fixed $\mu^k$ is globally convergent given that the condition numbers of the Hessians $\nabla^2 f (x^k)$ are uniformly bounded above (see \cite{More1984} and \cite[Section 3.4]{Nesterov2006}).  In the literature, $\mu^k$ has also commonly been expressed in terms $\| \nabla f(x^k) \|$ as done in Polyak~\cite{Polyak2009} who proves global convergence for convex functions and establishes $O(\epsilon^{-4})$ global convergence rate provided that the Hessian is Lipschitz continuous and that the level set of $f$ at the initial point $x^0$ is compact. Ueda and Yamashita~\cite{Ueda2010} extended Polyak's results to nonconvex functions and suggest $\mu^k \propto \| \nabla f(x^k) \|^{\delta}$. They show that with $\delta \leq 1/2$ the global convergence rate improves to $O(\epsilon^{-2})$. Recently, Mishchenko~\cite{Mishchenko2023} improved these results showing that with the same assumptions we can achieve $O(\epsilon^{-1/2})$ rate with $\mu^k = \sqrt{ (L/2) \| \nabla f(x^k) \|}$ where $L$ is the Lipschitz constant. Below we state Mishchenko's result formally without proof. 

\begin{theorem}\cite[Theorem 1]{Mishchenko2023} \label{thm:LM_global_convergence}
    Let $f : \mathbb{R}^n \rightarrow \mathbb{R}$ be convex and let the following assumptions hold
    \begin{enumerate}[label=Assumption \arabic*., leftmargin=*]
        \item the Hessian of $f$ is $L$-Lipschitz.
        \item the objective function $f$ has a finite optimum $x^*$ such that $f(x^*) = \min_{x \in \mathbb{R}^n} f(x)$, and the level set $\{ x \in \mathbb{R}^n : f(x) \leq f(x^0) \}$ is bounded.
    \end{enumerate}
    Then, Newton method with Levenberg-Marquardt regularization with $\mu^k = \sqrt{L\| \nabla f(x^k) \|/2}$\\ and $\alpha^k = 1$ for $k \geq 0$ achieves $f(x^k) - f(x^*) < \epsilon$ in at most $O(\epsilon^{-1/2})$ iterations.
\end{theorem}

\subsection{Global convergence and its rate}
Using \Cref{thm:LM_global_convergence} and the fact that the solution sets of \eqref{eq:merit_function} and the primal-dual pair \eqref{eq:primal} and \eqref{eq:dual} coincide when \eqref{eq:primal} and \eqref{eq:dual} have an optimal solution, we can show that Newton method with Levenberg-Marquardt regularization (\Cref{alg:Newton_with_LM}) provides a globally convergent scheme for finding an optimal solution to \eqref{eq:primal} and \eqref{eq:dual}:
\begin{theorem}\label{thm:MF_global_convergence}
    Suppose that the primal-dual pair \eqref{eq:primal} and \eqref{eq:dual} have an optimal solution and that the set of optimal solutions to \eqref{eq:primal} and \eqref{eq:dual} is bounded. Then, minimizing the merit function $f_3(x,\lambda, s)$ defined in \eqref{eq:merit_function} using Newton method with Levenberg-Marquardt regularization (\Cref{alg:Newton_with_LM}) with $\mu^k = \sqrt{\| \nabla f(x^k) \|/2}$ and $\alpha^k = 1$ for $k \geq 0$ achieves 
    \begin{align}
    \| c^T x^k - b^T \lambda^k \| &\leq \epsilon, 
    &&\label{eq:iterate_less_epsilon_1} \\
    \| A x^k - b \| &\leq \epsilon, 
    &&\label{eq:iterate_less_epsilon_2} \\
    \| A^T \lambda^k + s^k - c \| &\leq \epsilon, 
    &&\label{eq:iterate_less_epsilon_3} \\
    x^k_j &\geq -\epsilon \quad \text{for } j = 1, \dots, n, 
    &&\label{eq:iterate_less_epsilon_4} \\
    s^k_j &\geq -\epsilon \quad \text{for } j = 1, \dots, n 
    &&\label{eq:iterate_less_epsilon_5}
\end{align}
    in at most $O(\epsilon^{-3/2})$ iterations.
\end{theorem}
\begin{proof}
    In \Cref{lemma:Hessian_is_Lipschitz} we showed that $f_q(x,\lambda, s)$ for $q = 3$ is $1$-Lipschitz, so the first assumption in \Cref{thm:LM_global_convergence} is satisfied for $f_3(x,\lambda, s)$ and $L = 1$. 

    Now, suppose that the primal-dual pair \eqref{eq:primal} and \eqref{eq:dual} have an optimal solution and that the set of optimal solutions to \eqref{eq:primal} and \eqref{eq:dual} is a bounded set. The solution sets of \eqref{eq:merit_function} and the primal-dual pair \eqref{eq:primal} and \eqref{eq:dual} coincide and hence the set of minimizers of $f_3(x,\lambda, s)$ is also a bounded set. This together with the fact that $f_3(x,\lambda, s)$ is bounded below by zero imply that $f_3(x,\lambda, s)$ has a finite optimum $(x^*, \lambda^*, s^*)$. Moreover, since $f_3(x,\lambda, s)$ is convex, the level set $\{ (x,\lambda, s) \in \mathbb{R}^{2n+m} : f_3(x,\lambda, s) \leq f_3(x^0,\lambda^0, s^0) \}$ is bounded for all $(x^0,\lambda^0, s^0) \in \mathbb{R}^{2n+m}$ (see \cite{Polyak2009}). \Cref{thm:LM_global_convergence} says that we have
    \begin{equation}\label{eq:f_3^k_minus_f_3_at_opt}
        f_3(x^k,\lambda^k, s^k) - f_3(x^*, \lambda^*, s^*) \leq \epsilon
    \end{equation}
    in at most $O(\epsilon^{-1/2})$ iterations. By \Cref{thm:equivalence}, we have $f_3(x^*, \lambda^*, s^*) = 0$ and hence \eqref{eq:f_3^k_minus_f_3_at_opt} becomes $f_3(x^k,\lambda^k, s^k) \leq \epsilon$. Since each term of $f_3(x^k,\lambda^k, s^k)$ is nonnegative, each term of $f_3(x^k,\lambda^k, s^k)$ must be less than $\epsilon$. In other words,
    \begin{equation}
        \begin{aligned}
            \| c^T x^k - b^T \lambda^k \| & \leq \sqrt{2 \epsilon} \\
            \| Ax^k-b \| & \leq \sqrt{2 \epsilon} \\
            \|A^T\lambda^k + s^k - c\| & \leq \sqrt{2 \epsilon} \\
            \max\{-x^k_j,0\} & \leq \sqrt[3]{q(q-1)\epsilon} \;\; \text{for $j = 1, \dots, n$} \\
            \max\{-s^k_j,0\} & \leq \sqrt[3]{q(q-1)\epsilon} \;\; \text{for $j = 1, \dots, n$}
        \end{aligned}
    \end{equation}
    is achieved in $O(\epsilon^{-1/2})$ iterations. Hence, \eqref{eq:iterate_less_epsilon_1}, \eqref{eq:iterate_less_epsilon_2} and \eqref{eq:iterate_less_epsilon_3}  are achieved in $O(\epsilon^{-1})$ iterations and \eqref{eq:iterate_less_epsilon_4}, \eqref{eq:iterate_less_epsilon_5} are achieved in $O(\epsilon^{-3/2})$ iterations.
\end{proof}
The assumption in \Cref{thm:LM_global_convergence} that the optimal set of \eqref{eq:primal} and \eqref{eq:dual} has to be bounded includes the common scenario of \eqref{eq:primal} and \eqref{eq:dual} having a unique solution. Furthermore, note that the number of iterations is independent of problem dimension. If at each iteration a linear system is solved in $O(m^3) + O(n^3)$ time (see \Cref{sec:linear_system}) then the total complexity of finding an optimal solution to \eqref{eq:primal} and \eqref{eq:dual} using \Cref{alg:Newton_with_LM} is $O((m^3+n^3) \epsilon^{-3/2})$. 

\subsection{Solving the linear system} 
\label{sec:linear_system}
Solving the linear system \eqref{eq:mf_update} accounts for most of the computational effort in \Cref{alg:Newton_with_LM}. The linear system has size $(2n+m) \times (2n+m)$, but we can reduce it to solving two linear systems of sizes $n\times n$ and $m \times m$. The linear system \eqref{eq:mf_update} with a more general diagonal perturbation of the Hessian can be written as

\begin{equation}
\begingroup\footnotesize
   \left( \begin{pmatrix}
    cc^T & -cb^T & 0 \\
    -bc^T & bb^T & 0 \\
    0 & 0 & 0 
\end{pmatrix} + 
    \begin{pmatrix}
    A^TA+D_1 & 0 & 0 \\
    0 & AA^T+D_2 & A \\
    0 & A^T & I + D_3
\end{pmatrix} \right)
\begin{pmatrix}
        \Delta x \\
        \Delta \lambda \\
        \Delta s
    \end{pmatrix} = \begin{pmatrix}
        -\nabla_x f \\
        -\nabla_{\lambda} f \\
        -\nabla_{s} f
    \end{pmatrix},
\endgroup
\end{equation}
where $D_1$, $D_2$ and $D_3$ are diagonal matrices with positive diagonal entries. We can express $\Delta s$ via $\Delta \lambda$ and remove it from the system:
\begin{equation}
\begingroup\footnotesize
\begin{aligned}
       \left( \begin{pmatrix}
    cc^T & -cb^T \\
    -bc^T & bb^T 
\end{pmatrix} + 
    \begin{pmatrix}
    A^TA+D_1 & 0 \\
    0 & AD_3(I+D_3)^{-1}A^T+D_2 
\end{pmatrix} \right)
\begin{pmatrix}
        \Delta x \\
        \Delta \lambda 
    \end{pmatrix} & = \begin{pmatrix}
        -\nabla_x f \\
        -\nabla_{\lambda} f + A(I+D_3)^{-1}\nabla_s f
    \end{pmatrix} \\
    \Delta s & = (I+D_3)^{-1} (-\nabla_s f - A^T \Delta \lambda),
\end{aligned}
\endgroup
\end{equation}
where, during algebraic manipulations, we have used the identity $I - D_3(I+D_3)^{-1} = (I+D_3)^{-1}$. Now, the system is block diagonal plus a rank-one update. We can solve this system by finding Cholesky decompositions of $A^TA+D_1$ and $AD_3(I+D_3)^{-1}A^T+D_2$ separately and then apply a Cholesky rank one update algorithm to the whole $(m+n)\times (m+n)$ block diagonal matrix \cite{Gill1974}. The complexity of the former is $O(m^3+n^3)$ and the complexity of the latter is $O((m+n)^2)$.

\section{Numerical considerations and results} \label{sec:numerics}

We begin this section with an example that illustrates the effects of singular Hessians on the contours of $f_q$ near the minimizer. In \Cref{sec:alg_enhacement}, we propose a slightly different merit function whose Hessians are nonsingular. With this new merit function we propose a heuristic approach that numerically performs better. In \Cref{sec:optimal_lp} and \Cref{sec:unb_lin_prog}, we test the performance of the old and new merit functions on synthetically generated random linear programs of varying dimensions. Our test set includes three linear programs with optimal solutions of dimensions $(m,n) = (100,150)$, $(200,300)$, 
 $(500,750)$ and one unbounded problem of dimension $(m,n) = (50,150)$. 

\subsection{A simple example}\label{sec:challenges}
Let us consider a simple two-dimensional function that resembles the form of merit function $f_q$:     
\begin{equation}\label{eq:mf_ex}
    (x_1+x_2)^2 + \max\{-x_1,0\}^3 + \max\{-x_2,0\}^3.
\end{equation}
The contour levels of this function are shown in the left plot in \Cref{fig:4.1}. First, note straight contour lines in the first quadrant, indicating that the Hessian is singular in that region. Moreover, we see banana-shaped contours near the optimal solution $x^*$; these are known to cause difficulties (recall Rosenbrock function \cite{Rosenbrock1960}) even for Newton's method as its quadratic model may not approximate the function well near the solution. 

The right plot in \Cref{fig:4.1} draws $\max\{-x, 0\}^3$ against its quadratic approximation near a point slightly less than zero. We see that the quadratic model approximates $\max\{-x, 0\}^3$ well in the neighbourhood of zero for $x<0$, but does not yield a good approximation for $x>0$.   
\begin{figure}[t]
    \centering
    \includegraphics[width=0.8\textwidth]{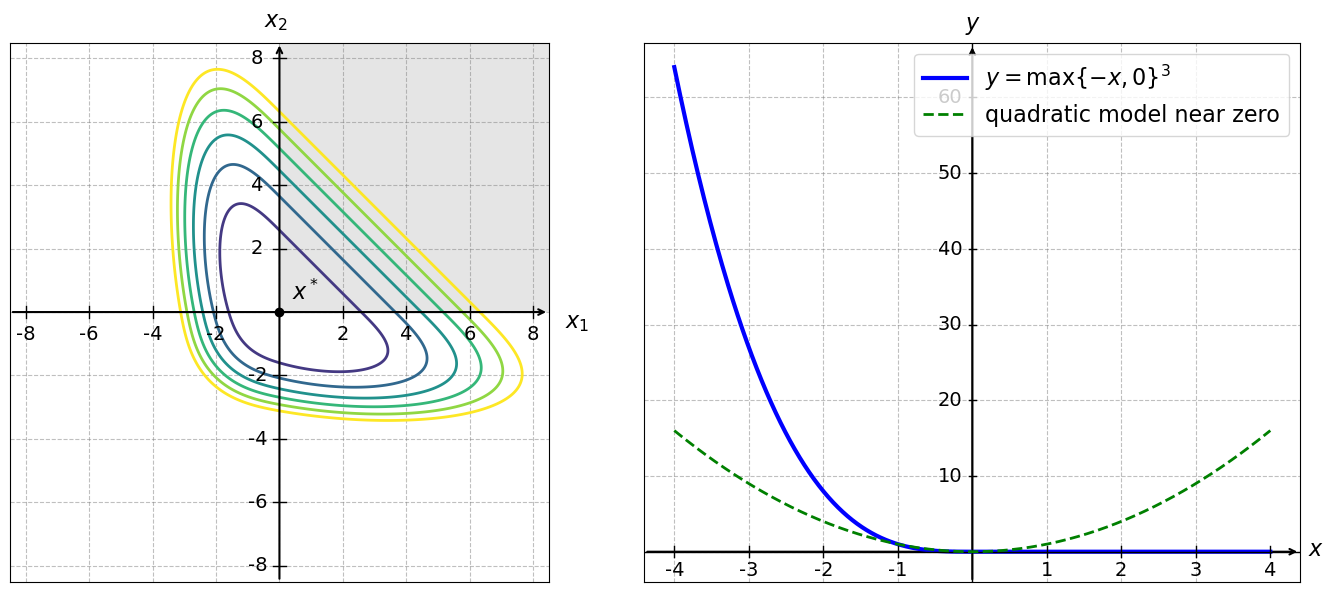} 
    \caption{Left: A contour plot of \eqref{eq:mf_ex}. Right: Illustration of $\max\{-x,0\}^3$ and its quadratic model near a point slightly less than zero.}
    \label{fig:4.1} 
\end{figure}

\subsection{The modified merit function and heuristic approach}\label{sec:alg_enhacement}
We propose the following modified merit function to address the challenges posed by formulation \eqref{eq:merit_function}: 
\begin{equation}\label{eq:homotopy_merit_function}\tag{HMF}
    \mathop{\arg\min}_{(x,\lambda,s) \in \mathbb{R}^{2n+m}}\, h_{q,\nu}(x,\lambda,s) := 
    f_q(x,\lambda,s) + \nu \| \lambda \|_2^2 + \frac{\nu}{q(q-1)}\sum_{j=1}^n \left( \max\{x_j,0\}^q + \max\{s_j,0\}^q \right),
\end{equation}
where $\nu > 0$ and where $f_q(x,\lambda,s)$ is the original merit function defined in \eqref{eq:merit_function}.
\begin{algorithm}[t]
\caption{Newton method with the modified merit function}
\label{alg:Newton_for_HMF}
\begin{algorithmic}[1]
\State Initialize $x^0, s^0 \in \mathbb{R}^n$, $\lambda^0 \in \mathbb{R}^m$ and $\nu^0 > 0$, $\theta \in \left(0, 1\right)$, $\mu^0 > 0$.
\For{$k = 0, 1, \dots$}
    \State Let $\nabla^2 h^k := \nabla^2 h_{q, \nu^k}(x^k, \lambda^k, s^k)$ and $\nabla h^k := \nabla h_{q, \nu^{k}}(x^k, \lambda^k, s^k)$, where $h_{q, \nu}$ is defined in \eqref{eq:homotopy_merit_function}.
    \State Update
    \begin{equation}\label{eq:hmf_update}
        \begin{pmatrix}
        x^{k+1} \\
        \lambda^{k+1} \\
        s^{k+1}
    \end{pmatrix}
    =
    \begin{pmatrix}
        x^k \\
        \lambda^k \\
        s^k
    \end{pmatrix}
    - \alpha^k (\nabla^2 h^k + \mu^k I)^{-1} \nabla h^k,
    \end{equation}
    \State where $\alpha^k$ is equal to $1$ or chosen using a line search method such as Armijo backtracking.
    \State Update $\nu^{k+1} = \theta \nu^k$.
    \State Update $\mu^{k+1}$.
\EndFor
\end{algorithmic}
\end{algorithm}
\begin{figure}[t]
    \centering
    \includegraphics[width=0.8\textwidth]{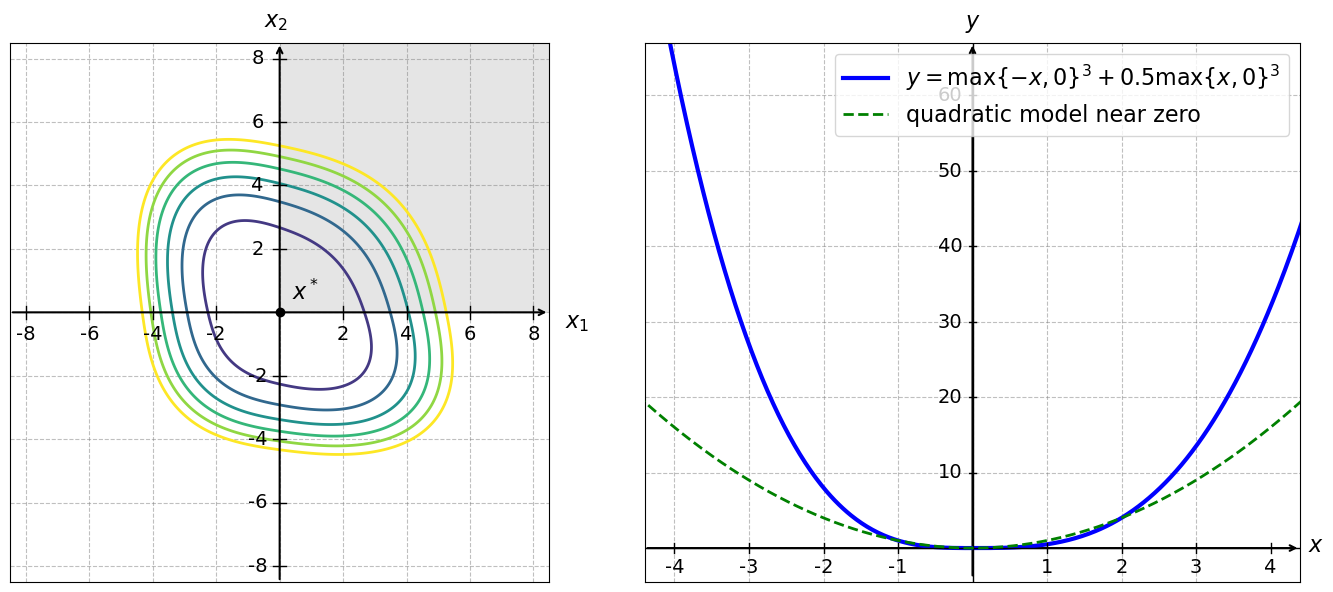} 
     \caption{Left: A contour plot of \eqref{eq:hmf_ex}. Right: Illustration of $\max\{-x,0\}^3 + 0.5\max\{x, 0\}^3$ and its quadratic model near zero.}
    \label{fig:4.2} 
\end{figure}

Most importantly, the Hessian of this new merit function is positive definite for all $(x,\lambda,s) \in \mathbb{R}^{2n+m}$ when $\nu > 0$. This new merit function can be thought of as an approximation to $f_q$. The approximation is better for smaller $\nu$ and, at $\nu = 0$, $h_{q, \nu}$ is identically equal to $f_q$. Instead of solving \eqref{eq:merit_function}, we propose to solve \eqref{eq:homotopy_merit_function} following the scheme of \Cref{alg:Newton_with_LM}, but with gradually decreasing $\nu$ at each subsequent iteration (see \Cref{alg:Newton_for_HMF}). We borrow the idea from homotopy methods \cite{Nazareth1986, Nazareth1991, Nocedal2006, Wright1997}: as $v$ approaches $0$, the new merit function $h_{q,\nu}$ approximates $f_q$ better and, hopefully, the iterates are gradually driven towards the optimal set of $f_q$. At each iteration we solve a nonsingular Newton system to find a suitable step that minimizes $h_{q,\nu^{k}}$ and then we set $\nu^{k+1} = \theta \nu^k$ where $\theta$ is strictly less than 1 to drive $\nu^k$'s to zero. For initially large values of $\nu$, contour lines of \eqref{eq:homotopy_merit_function} are near elliptical, but they gradually become banana-shaped as $\nu$ gets smaller. By that stage, we hope that the iterates are in the vicinity of the optimal set and fast local convergence takes over (see \Cref{fig:4.3}).

\begin{rem}
   In \eqref{eq:hmf_update}, we added the regularization term $\mu^k I$ to the Hessian (with $\mu^k \sim 10^{-9}$ in the experiments) as a safeguard against ill-conditioning. As $\nu^k$ becomes very small in the final iterations, the Hessian $\nabla^2 h^k$ of the modified merit function tends to become increasingly ill-conditioned, necessitating some regularization to maintain numerical stability.   
\end{rem}

\paragraph{A simple example revisited.} Let us revisit the function in \eqref{eq:mf_ex} but with additional max functions to resemble the form of \eqref{eq:homotopy_merit_function}. We would like to see the changes in contour lines and the quadratic approximation of max functions. Consider
\begin{equation}\label{eq:hmf_ex}
    (x_1+x_2)^2 + (\max\{-x_1,0\}^3 + \max\{-x_2,0\}^3) + 0.5(\max\{x_1,0\}^3 + \max\{x_2,0\}^3).
\end{equation}
The contour lines shown in \Cref{fig:4.2} are now close to elliptical -- a significant improvement over banana-shaped contours and more tractable for Newton's method. The right plot in \Cref{fig:4.2} also shows a better quadratic approximation of the max functions of \eqref{eq:hmf_ex} near zero.

\subsection{Linear programs with an optimal solution}\label{sec:optimal_lp}
We test the performance of \Cref{alg:Newton_with_LM} and \Cref{alg:Newton_for_HMF} on synthetically generated random linear programs of dimensions $(m,n) = (100,150)$, $(200,300)$, $(500,750)$. All three problems have optimal solutions. We test two variants of \Cref{alg:Newton_with_LM} and one variant of \Cref{alg:Newton_for_HMF} with parameters outlined in the following table.
\begin{table}[H]
\centering
\begin{tabular}{l|c|c|c|c}
 & value of $q$ & regularization parameter & line search & value of $\theta$ \\
\hline
\Cref{alg:Newton_with_LM}a & $3.0$ & $\mu^k = \sqrt{(1/2)\|\nabla f(x^k)\|}$ & $\alpha^k = 1$ for all $k\geq 0$ & N/A \\
\Cref{alg:Newton_with_LM}b & $2.1$ & $\mu^k = 10^{-9}$ for all $k\geq 0$ & Armijo backtracking & N/A \\
\Cref{alg:Newton_for_HMF} & $2.1$ & $\mu^k = 10^{-9}$ for all $k\geq 0$ & Armijo backtracking & $0.8$ \\
\end{tabular}
\caption{The set of parameters used in the experiments for \Cref{alg:Newton_with_LM} and \Cref{alg:Newton_for_HMF}.}
\label{tab:parameters}
\end{table}
The results of the experiments are provided in \Cref{fig:4.3}. The leftmost plots show the results of \Cref{alg:Newton_with_LM}a for the adaptive choice of $\mu^k$. We observe slow sublinear convergence aligning with the theoretical worst-case sublinear rate derived in \Cref{thm:MF_global_convergence}. \Cref{alg:Newton_with_LM}b (the middle plots) and \Cref{alg:Newton_for_HMF} (the rightmost plots) perform better. Interestingly, \Cref{alg:Newton_with_LM}b exhibits slow initial convergence followed by fast convergence resembling the behaviour of the Newton method for well-conditioned problems. We see from the plots that for problems of size $(m,n) = (100, 150)$ and $(200, 300)$ the difference between \Cref{alg:Newton_with_LM}b and \Cref{alg:Newton_for_HMF} is not significant. Nonetheless, we can obtain faster convergence with \Cref{alg:Newton_for_HMF} for these problems if we set $\theta$ to a smaller value that will drive $\nu$ to zero faster. However, setting $\theta$ to a value that is too small may result in stalling of the progress similar to the initial slow convergence phase of \Cref{alg:Newton_with_LM}b. Therefore, there is a need for an adaptive definition of $\nu^k$ based on the progress of the iterates. 

In \Cref{tab:xres_all}, we provide the relative errors of $x^k$ with respect to the primal optimal solution $x^*$ of the last ten iterations of \Cref{alg:Newton_with_LM}b and \Cref{alg:Newton_for_HMF}. Implementations of \Cref{alg:Newton_with_LM} and \ref{alg:Newton_for_HMF} on a single random problem are available on GitHub\footnote{\url{https://github.com/alinaabdikarimova/lp-via-unconstrained-minimization}}.


\begin{figure}[t]
    \centering
    \includegraphics[width=\textwidth]{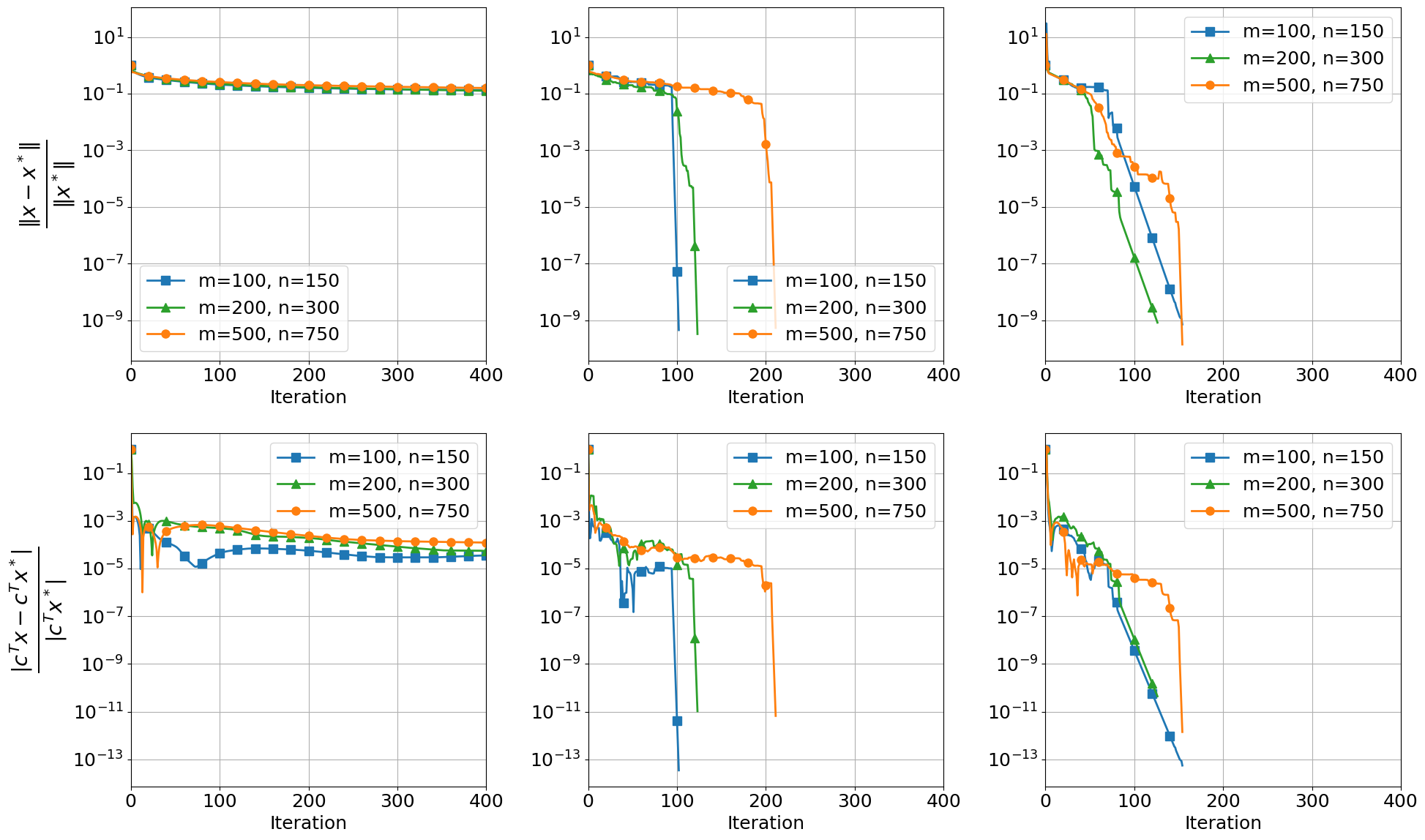} 
    \caption{This figure shows results for \Cref{alg:Newton_with_LM}a (leftmost plots), \Cref{alg:Newton_with_LM}b (middle plots) and \Cref{alg:Newton_for_HMF} (rightmost plots) when tested on three synthetically generated random linear programs of dimensions $(m,n) = (100,150)$, $(200,300)$, $(500,750)$. The plots in the first row show the relative error of $x^k$ with respect to the primal optimal solution $x^*$. The plots in the second row show the relative error of $c^Tx^k$ with respect to the primal objective function value at the primal optimal solution.}
    \label{fig:4.3} 
\end{figure}

\begin{table}[h!]
\centering
\footnotesize
\begin{minipage}{0.3\textwidth}
\centering
\begin{tabular}{|c|c|}
\hline
\Cref{alg:Newton_with_LM}b & \Cref{alg:Newton_for_HMF} \\
\hline
$1.8{\times}10^{-1}$ & $4.5{\times}10^{-9}$ \\
$1.8{\times}10^{-1}$ & $4.1{\times}10^{-9}$ \\
$8.2{\times}10^{-3}$ & $3.0{\times}10^{-9}$ \\
$7.4{\times}10^{-4}$ & $2.4{\times}10^{-9}$ \\
$6.8{\times}10^{-5}$ & $2.0{\times}10^{-9}$ \\
$6.2{\times}10^{-6}$ & $1.6{\times}10^{-9}$ \\
$5.7{\times}10^{-7}$ & $1.3{\times}10^{-9}$ \\
$5.3{\times}10^{-8}$ & $1.2{\times}10^{-9}$ \\
$4.9{\times}10^{-9}$ & $1.1{\times}10^{-9}$ \\
$4.5{\times}10^{-10}$ & $7.1{\times}10^{-10}$ \\
\hline
\end{tabular}
\caption*{$m = 100, n = 150$}
\end{minipage}
\hfill
\begin{minipage}{0.3\textwidth}
\centering
\begin{tabular}{|c|c|}
\hline
\Cref{alg:Newton_with_LM}b & \Cref{alg:Newton_for_HMF} \\
\hline
$5.6{\times}10^{-5}$ & $5.1{\times}10^{-9}$ \\
$5.5{\times}10^{-5}$ & $4.2{\times}10^{-9}$ \\
$5.5{\times}10^{-5}$ & $3.4{\times}10^{-9}$ \\
$4.9{\times}10^{-5}$ & $2.8{\times}10^{-9}$ \\
$4.9{\times}10^{-5}$ & $2.3{\times}10^{-9}$ \\
$4.5{\times}10^{-6}$ & $1.9{\times}10^{-9}$ \\
$4.1{\times}10^{-7}$ & $1.5{\times}10^{-9}$ \\
$3.8{\times}10^{-8}$ & $1.2{\times}10^{-9}$ \\
$3.5{\times}10^{-9}$ & $1.0{\times}10^{-9}$ \\
$3.3{\times}10^{-10}$ & $8.2{\times}10^{-10}$ \\
\hline
\end{tabular}
\caption*{$m = 200, n = 300$}
\end{minipage}
\hfill
\begin{minipage}{0.3\textwidth}
\centering
\begin{tabular}{|c|c|}
\hline
\Cref{alg:Newton_with_LM}b & \Cref{alg:Newton_for_HMF} \\
\hline
$4.1{\times}10^{-4}$ & $6.3{\times}10^{-6}$ \\
$1.5{\times}10^{-4}$ & $6.3{\times}10^{-6}$ \\
$7.4{\times}10^{-5}$ & $3.0{\times}10^{-6}$ \\
$7.4{\times}10^{-5}$ & $3.0{\times}10^{-6}$ \\
$7.4{\times}10^{-5}$ & $3.0{\times}10^{-6}$ \\
$6.9{\times}10^{-6}$ & $1.6{\times}10^{-6}$ \\
$6.4{\times}10^{-7}$ & $1.6{\times}10^{-7}$ \\
$6.0{\times}10^{-8}$ & $1.5{\times}10^{-8}$ \\
$5.6{\times}10^{-9}$ & $1.4{\times}10^{-9}$ \\
$5.4{\times}10^{-10}$ & $1.4{\times}10^{-10}$ \\
\hline
\end{tabular}
\caption*{$m = 500, n = 750$}
\end{minipage}

\caption{Relative errors of $x^k$ with respect to the primal optimal solution $x^*$ of the last ten iterations of \Cref{alg:Newton_with_LM}b and \Cref{alg:Newton_for_HMF}.}
\label{tab:xres_all}
\end{table}

\subsection{An unbounded linear program}\label{sec:unb_lin_prog}
We run \Cref{alg:Newton_with_LM}a, \Cref{alg:Newton_with_LM}b and \Cref{alg:Newton_for_HMF} on a randomly generated unbounded linear program with the same parameters provided in \Cref{tab:parameters}. Our synthetic unbounded problem has dimensions $(m,n) = (50,150)$. We first solved this problem with Python's HiGHS solver \cite{Huangfu2023} to verify unboundedness. The results of the experiment are provided in \Cref{fig:4.4}. 

From all three plots in \Cref{fig:4.4} we observe that $\| \nabla f_q^k \|$ approaches zero, so the algorithms converge to an optimal solution. Since the values of $f_q(x^k,\lambda^k,s^k)$ remain bounded away from zero, from \Cref{cor:no_opt_sol} it follows that the problem is either infeasible or unbounded. Since $\|Ax^k - b\|$ approaches zero, indicating feasibility of the primal problem, we conclude that the problem is unbounded.
\begin{figure}[t]
    \centering
    \includegraphics[width=\textwidth]{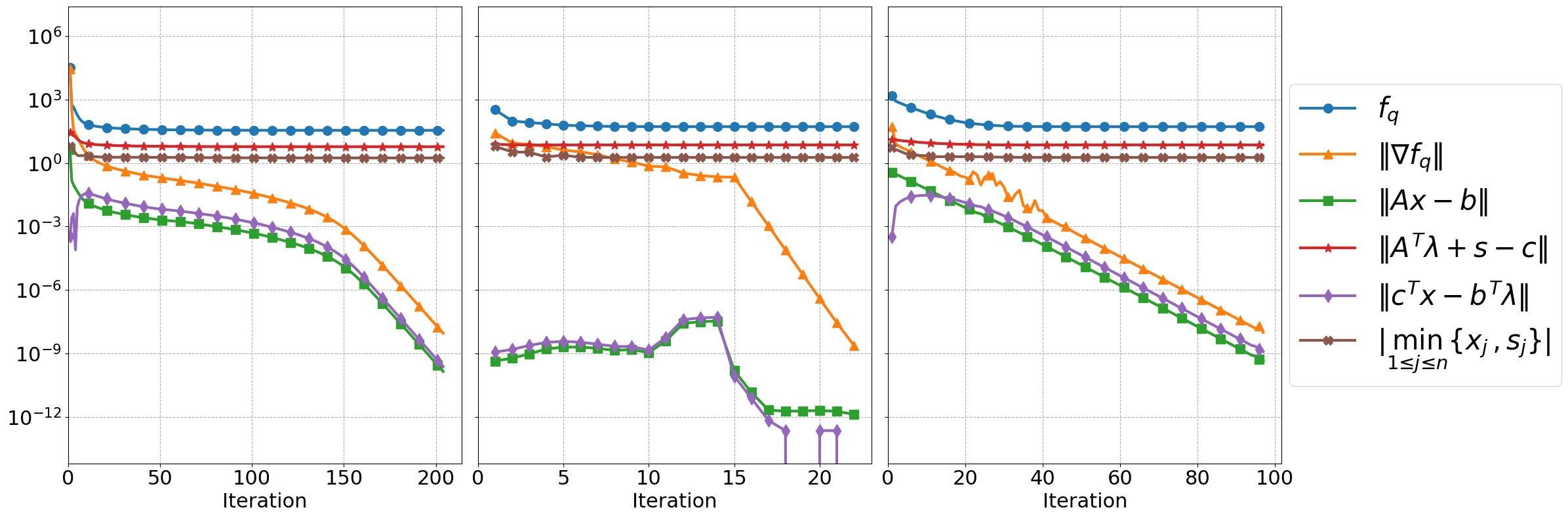} 
    \caption{This figure shows results for Algorithm 1a (leftmost plot), Algorithm 1b (middle plot) and Algorithm 2 (rightmost plot) when tested on a synthetically generated random unbounded linear program of dimension $(m,n) = (50,150)$. }
    \label{fig:4.4} 
\end{figure}


\section{Summary} \label{sec:conclusion}


In this paper, we have demonstrated that solving primal-dual linear program \eqref{eq:primal} and \eqref{eq:dual} is equivalent to finding an unconstrained minimizer of a convex and twice continuously differentiable merit function $f_q(x,\lambda,x)$ defined in \eqref{eq:merit_function}. The merit function $f_q(x,\lambda,x)$ is a sum of i) the squared duality gap $c^Tx-b^T\lambda$ to encourage optimality, ii) the squared equality constraint terms $Ax-b$ and $A^T\lambda + s -c$ to encourage primal and dual feasibility, and iii) the nonnegativity penalty terms $\max\{-x_j,0\}^q$ and $\max\{-s_j,0\}^q$ to encourage nonnegative solutions, where $q>2$ to ensure second-order differentiability. From the optimality conditions \eqref{eq:optimality_conditions} for linear programs, it follows that $(x^*,\lambda^*,x^*)$ is a solution to the primal-dual linear program if and only if $(x^*,\lambda^*,x^*)$ is a minimizer of the merit function $f_q(x,\lambda,x)$  and $f_q(x^*,\lambda^*,x^*) = 0$. In other words, when the primal-dual pair have an optimal solution, its unconstrained reformulation \eqref{eq:merit_function} can be viewed as a zero residual problem. Moreover, the merit function is defined over the entire $\mathbb{R}^{2n+m}$ allowing iterates to be infeasible with respect to \eqref{eq:primal} and \eqref{eq:dual}. 

We saw that minimizing $f_q(x,\lambda,x)$ is not trivial due to its Hessian being singular at the optimal solutions and at some other points in the domain (for example, for $(x,s) > 0$). To handle singular Hessians, we regularized Newton with Levenberg-Marquardt approach. We proved that this variant of the regularized Newton method with a particular choice of the regularization parameter $\mu^k$ (see \Cref{thm:MF_global_convergence}) converges globally to an optimal solution of the primal-dual pair at $O(\epsilon^{-3/2})$ rate requiring only the assumption that the optimal set of the primal-dual linear program is bounded. Note that this assumption includes the common scenario of the linear program having a unique solution.

Before conducting the numerical experiments, we illustrated the effects of singular Hessians on the contours of the merit function $f_q$ and introduced a slightly modified merit function $h_{q,\nu}$ defined in \eqref{eq:homotopy_merit_function} with nonsingular Hessians and an improved landscape to help Newton to converge faster. We propose a heuristic method based on the new merit function, in which the iterates are updated via Newton steps while $\nu$ is gradually reduced to zero, allowing $h_{q,\nu}$ to more closely approximate $f_q$ and ultimately leading to convergence to the optimal solution.

We conducted numerical experiments on synthetically generated random linear programs of varying dimensions. We minimized the merit function $f_q$ with two variants of the Newton method with Levenberg-Marquardt regularization corresponding to two different ways of defining regularization parameter $\mu^k$ (i.e., adaptive and constant $\mu^k$) and minimized the new merit function $h_{q,\nu}$. On problems with optimal solutions, we observed convergence to very high accuracy for both merit functions, except for the variant with adaptive $\mu^k$. Between the two successful variants, both merit functions exhibited similar performance on problems of size $(m,n) = (100, 150)$ and $(200, 300)$. However, for the higher-dimensional case $(500, 750)$, the new merit function $h_{q,\nu}$ achieved faster convergence.

\section{Future work} \label{sec:future_work}
Our work can be improved in multiple ways. 
\begin{itemize}

    \item[-] \textbf{Adaptive update of parameter $\nu$.}  
    \Cref{alg:Newton_for_HMF} minimizes the modified merit function in \eqref{eq:homotopy_merit_function}, which depends on a parameter $\nu > 0$ that is reduced at each iteration by a fixed factor $\theta$. Experiments indicate that for small-dimensional problems, choosing a smaller $\theta$ can accelerate convergence, whereas for large-dimensional problems, $\nu$ should be decreased more gradually to prevent stalling. In our experiments, $\theta$ was tuned via trial and error; for general problems, however, a more sophisticated adaptive strategy would be preferable.
    
    
    
    
    

    \item[-] \textbf{Improved convergence rate.} Our numerical experiments demonstrate that we can do much better in practice than the worst-case bound derived in \Cref{thm:MF_global_convergence}, especially at the final iterations when the iterates are close to the optimal set. Regularized Newton method is well-known to converge globally to the optimal set of a convex function \cite{Polyak2009} with global convergence rates derived, for example, in \cite{Mishchenko2023, Polyak2009, Ueda2010}. Local convergence rates have also been studied. It was proven that (see, e.g., \cite{Dan2002, Li2004}), under the so-called local error bound assumption, regularized Newton method inherits classical Newton's local quadratic convergence property. It is very possible that, with very mild assumptions, the merit function in \eqref{eq:merit_function} or its modification satisfies the local error bound.

    \item[-] \textbf{Cheap iterations.} Each iteration in \Cref{alg:Newton_with_LM} and \ref{alg:Newton_for_HMF} require a solution of a linear system to find the Newton's search direction. One may consider computing the search directions approximately to cheapen each iteration's cost (see \cite{Dan2002, Li2004, Li2009}). 
    Dan et al.~\cite{Dan2002} showed that an inexact regularized Newton method is globally convergent under mild assumptions and, under the additional assumption of the local error bound, enjoys local superlinear convergence to the optimal set. 
    
        
    \end{itemize}
Of course, there are many more avenues that we can explore to enhance and extend this work. One of the important possible extensions of this work concerns unconstrained formulations of more complex constrained optimization problems. The fact that complementary slackness in linear programs can be replaced by the duality gap, that is linear in the variables of the problem, allowed us to formulate a convex merit function. If we consider KKT conditions for more complex problems and use a similar approach to put them in one unconstrained formulation, what properties would the resulting merit function have? How easy would it be to find its solutions?

Consider, for example, a nonnegative quadratic programming problem:
\begin{equation}\label{eq:quadratic}
    \text{min} \;\; \frac{1}{2} x^T Q x + c^Tx \;\; \text{subject to} \;\;  x  \geq 0.
\end{equation}
Its KKT conditions are given by
\begin{equation}\label{eq:kkt_quad}
    \begin{aligned}
        Qx + c - \lambda & = 0_n \\
        x & \geq 0 \\
        \lambda & \geq 0 \\
        x_j\lambda_j & = 0 \text{ for $j=1,\dots,n$.}
    \end{aligned}
\end{equation}
We can keep the complementary slackness conditions $x_j\lambda_j = 0$ as they are or replace them with $x^T\lambda = 0$, since $x$ and $\lambda$ are nonnegative. With the latter choice, a possible merit function for \eqref{eq:kkt_quad} is 
\begin{equation}\label{eq:merit_for_kkt}
    \min_{x,\lambda} \, m_q(x,\lambda) := \|Qx+c-\lambda\|_2^2 + \sum_{j=1}^n(\max\{-x_j,0\}^q + \max\{-\lambda_j,0\}^q) + (x^T\lambda)^2,
\end{equation}
where $q>2$ and where $(x^T\lambda)^2$ is nonconvex. Note that $(x^*,\lambda^*)$ is a solution to \eqref{eq:kkt_quad} if and only if $(x^*,\lambda^*)$ is a solution to \eqref{eq:merit_for_kkt} and $m_q(x^*,\lambda^*) = 0$ (an analogue of \Cref{thm:equivalence} for linear programs). Since KKT conditions \eqref{eq:kkt_quad} provide sufficient and necessary conditions for optimality, constrained problem \eqref{eq:quadratic} and unconstrained problem \eqref{eq:merit_for_kkt} are equivalent. 
\appendix

\section{Simplified homogeneous model's merit function}\label{sec:mf_for_hlf}
Here, we provide the merit function and derive the gradient and Hessian for the simplified homogeneous model \eqref{eq:hlf}. Define 
\begin{equation}
    \begin{aligned}
        \gamma & :=  c^T x - b^Ty + \kappa  \\
        \rho & := b\tau - Ax \\
        \sigma & := c\tau - A^Ty - s 
    \end{aligned}
\end{equation}
The merit function for the simplified homogeneous model is given by
\begin{equation} \label{eq:merit_function_hlf}\tag{MF1}
\begin{aligned}
    \arg \min_{x,\lambda, s} f_q(x, y, s, \tau, \kappa) := & \frac{1}{2}(c^Tx-b^T\lambda + \kappa)^2 + \frac{1}{2}||Ax-b\tau||_2^2 + \frac{1}{2}|| A^T\lambda + s -c\tau ||_2^2 \\ 
    & +\frac{1}{q(q-1)} \Bigg( \max\{-\tau,0\}^q + \\
    & + \max\{-\kappa,0\}^q + \sum_{j=1}^n \max\{-x_j,0\}^q + \max\{-s_j,0\}^q \Bigg),
\end{aligned}
\end{equation}
The gradient and Hessian of $f_q(x, y, s, \tau, \kappa)$ are given by 
\begin{equation}
    \nabla f_q(x, y, s, \tau, \kappa) = \gamma \begin{pmatrix}
    c \\ -b \\ 0 \\ 0 \\ 1
\end{pmatrix} - \begin{pmatrix}
    A^T\rho \\ A\sigma \\ \sigma \\ w \\ 0
\end{pmatrix} - \frac{1}{q-1} \begin{pmatrix}
    (-x)_+^{q-1} \\ 0_m \\ (-s)_+^{q-1} \\ (-\tau)_+^{q-1} \\ (-\kappa)_+^{q-1}
\end{pmatrix},
\end{equation}
where $w = (Ax)^Tb + (A^Ty+s)^Tc - (\|b\|^2 + \|c\|^2)\tau$, and
\begin{equation}
    \nabla^2 f(x, y, s, \tau, \kappa) = H + \diag((-x)_+^{q-2}, 0_m , (-s)_+^{q-2},(-\tau)_+^{q-2},(-\kappa)_+^{q-2}),
\end{equation}
where  
\begin{equation}
    H = \begin{pmatrix}
    cc^T+A^TA & -cb^T & 0 & -A^Tb & c \\
    -bc^T & bb^T+AA^T & A & -Ac & -b \\
    0 & A^T & I & -c & 0 \\
    -b^TA & -c^TA^T & -c^T & \|b\|^2 + \|c\|^2 & 0 \\
    c^T & -b^T & 0 & 0 & 1
\end{pmatrix}.
\end{equation}

\newpage
\bibliography{main}
\bibliographystyle{plain}

\end{document}